\documentclass[10pt, a4paper, reqno]{amsart}
\usepackage{amscd}
\usepackage{dsfont}

\usepackage{amsmath}
\usepackage{amstext}
\usepackage{amsthm}
\usepackage{amssymb}
\usepackage{amsopn}
\usepackage{amssymb}
\usepackage{amsxtra}
\usepackage{amsfonts}
\usepackage{fancyhdr}
\usepackage{paralist, epsfig}
\usepackage[mathcal]{euscript}
\usepackage[english]{babel}
\usepackage[latin1]{inputenc}

\DeclareFontEncoding{FML}{}{}
\DeclareFontSubstitution{FML}{futm}{m}{it}

\DeclareSymbolFont{gletters}{FML}{futm}{m}{it}

\DeclareMathSymbol{\PSI}{\mathord}{gletters}  {32}
\renewcommand{\Psi}{\PSI}

\DeclareMathSymbol{\PHI}{\mathord}{gletters}   {30}
\renewcommand{\Phi}{\PHI}

\makeatletter
\g@addto@macro\normalsize{%
  \setlength\abovedisplayskip{10pt}
  \setlength\belowdisplayskip{10pt}
  \setlength\abovedisplayshortskip{5pt}
  \setlength\belowdisplayshortskip{8pt}
}

\allowdisplaybreaks

\newtheoremstyle{normal}
{5pt}
{5pt}
{\normalfont}
{}
{\bfseries}
{}
{0.4em}
{\bfseries{\thmname{#1}\thmnumber{ #2}.\thmnote{ \hspace{0.5em}(#3)\newline}}}

\newtheoremstyle{kursiv}
{5pt}
{5pt}
{\itshape}
{}
{\bfseries}
{}
{0.4em}
{\bfseries{\thmname{#1}\thmnumber{ #2}.\thmnote{ \hspace{0.5em}(#3)\newline}}}

\theoremstyle{kursiv}

\theoremstyle{normal}

\newtheorem{thm}{Theorem}
\newtheorem{ex}[thm]{Example}

\newtheorem{cor}[thm]{Corollary}
\newtheorem{lem}[thm]{Lemma}
\newtheorem{prop}[thm]{Proposition}

\renewcommand{\epsilon}{\varepsilon}

\renewcommand{\Re}{\operatorname{Re}\nolimits}

\newcommand{\gr}{\operatorname{Graph}\nolimits}

\newcommand{\pr}{\operatorname{pr}\nolimits}

\renewcommand{\bar}[1]{\overline{#1}}

\newtheoremstyle{PLAIN}{5pt}{5pt}{}{}{\bf}{.}{5pt}{\thmname{#1}\thmnumber{ #2}\thmnote{ #3}}
\theoremstyle{PLAIN}
\newtheorem*{thmA}{Theorem A}
\newtheorem*{thmB}{Theorem B}

\newcommand{\ran}{\operatorname{ran}\nolimits}

\newcommand{\Cnull}{\operatorname{C_0}\nolimits}

\usepackage{setspace}
\usepackage{color}

\definecolor{grey}{gray}{.3}

\begin{document}
$ $
\vspace{-30pt}

\title{Some remarks on the notions of\\boundary systems and boundary triple(t)s}

\author{Marcus Waurick\hspace{0.5pt}\MakeLowercase{$^{\text{1}}$} and Sven-Ake Wegner\hspace{0.5pt}\MakeLowercase{$^{\text{2}}$}}

\renewcommand{\thefootnote}{}
\hspace{-1000pt}\footnote{\hspace{5.5pt}2010 \emph{Mathematics Subject Classification}: Primary 47B25; Secondary 35G15.\vspace{1.6pt}}

\hspace{-1000pt}\footnote{\hspace{5.5pt}\emph{Key words and phrases}: Boundary triplet, boundary triple, boundary system, extension problem, deficiency index.\vspace{1.6pt}}


\hspace{-1000pt}\footnote{\hspace{0pt}$^{1}$\,University of Strathclyde, Department of Mathematics and Statistics, Livingstone Tower, 26 Richmond Street, Glasgow\linebreak\phantom{x}\hspace{1.2pt}G1\:1XH, Scotland, Phone: +44\hspace{1.2pt}(0)\hspace{1.2pt}141\hspace{1.2pt}/\hspace{1.2pt}548\hspace{1.2pt}-\hspace{1.2pt}3817, E-Mail: marcus.waurick@strath.ac.uk.\vspace{1.6pt}}

\hspace{-1000pt}\footnote{\hspace{0pt}$^{2}$\,Corresponding author: Teesside University, School of Science, Engineering and Design, Stephenson Building, Middles-\linebreak\phantom{x}\hspace{1.2pt}brough, TS1\;3BX, United Kingdom, phone: +44\,(0)\,1642\hspace{1.2pt}/\hspace{1.2pt}73\,-\,8200, E-Mail: s.wegner@tees.ac.uk.\vspace{1.6pt}}

\begin{abstract} In this note we show that if a boundary system in the sense of (Schubert et al.~2015) gives rise to any skew-self-adjoint extension, then it induces a boundary triplet and the classification of all extensions given by (Schubert et al.~2015) coincides with the skew-symmetric version of the classical characterization due to (Gorbachuk et al.~1991). On the other hand we show that for every skew-symmetric operator there is a natural boundary system which leads to an explicit description of at least one maximal dissipative extension. This is in particular also valid in the case that no boundary triplet exists for this operator.
\end{abstract}

\maketitle


\vspace{-10pt}
\section{Skew-self-adjoint extensions of skew-symmetric operators}\label{SEC:1}

\smallskip

Throughout this paper $\mathcal{H}$ denotes a Hilbert space and $H_0\colon D(H_0)\subseteq\mathcal{H}\rightarrow\mathcal{H}$ denotes a densely defined and closed linear operator. We say that $H_0$ is \textit{skew-symmetric} if $H_0\subseteq-H_0^{\star}$ holds. A densely defined operator $H\colon D(H)\subseteq\mathcal{H}\rightarrow\mathcal{H}$ is \textit{skew-self-adjoint} if $H=-H^{\star}$ holds. We point out, cf.~Picard et al.~\cite[footnote on p.~751]{P}, that in the literature the term \textit{skew-adjoint} seems to be more common although this \textquotedblleft{}binary expression\textquotedblright{} might cause confusion. Notice that a skew-self-adjoint extension $H$ of $H_0$ automatically satisfies $H\subseteq-H_0^{\star}$ and that a skew-self-adjoint restriction of $H_0^{\star}$ automatically satisfies $-H_0\subseteq H$. Moreover, $H$ is a skew-self-adjoint extension of $H_0$ if and only if $-H$ is a skew-self-adjoint extension of $-H_0$. Theorems~A and B below thus both deal with the question of determining the skew-self-adjoint extensions of $H_0$.

\medskip

Schubert et al.~\cite[Definition 3.1]{BS} defined the following notion. A quintuple $(\Omega,\mathcal{G}_1,\mathcal{G}_2,F,\omega)$ is a \textit{boundary system} for $H_0$, if $\Omega$ is a sesquilinear form on $\mathcal{H}\oplus\mathcal{H}$, $\mathcal{G}_1$ and $\mathcal{G}_2$ are Hilbert spaces, $\omega$ is a sesquilinear form on $\mathcal{G}_1\oplus\mathcal{G}_2$ and $F\colon\gr(H_0^{\star})\rightarrow\mathcal{G}_1\oplus\mathcal{G}_2$ is a surjective map such that
\begin{equation}\label{BS-1}
\Omega((x,H_0^{\star}x),(y,H_0^{\star}y)) = \omega(F(x,H_0^{\star}x),F(y,H_0^{\star}y))
\end{equation}
holds for all $x$, $y\in D(H_0^{\star})$. By \cite[Examples 2.7(b) and (c)]{BS} the \textit{standard symmetric form} on $\mathcal{H}\oplus\mathcal{H}$ is given by $\Omega((x,y),(u,v))=(x,v)_{\mathcal{H}}+(y,u)_{\mathcal{H}}$ and the \textit{standard unitary form} on $\mathcal{G}_1\oplus\mathcal{G}_2$ is given by $\omega((x,y),(u,v))=(x,u)_{\mathcal{G}_1}-(y,v)_{\mathcal{G}_2}$. This choice of a boundary system leads to the following characterization of extensions, see \cite[Theorem 3.6]{BS}.

\smallskip

\begin{thmA} Let $H_0$ be skew-symmetric and let $(\Omega,\mathcal{G}_1,\mathcal{G}_2,F,\omega)$ be a boundary system for $H_0$, where $\Omega$ is the standard symmetric form and $\omega$ is the standard unitary form. Then the operator $H\subseteq H_0^{\star}$ is skew-self-adjoint if and only if there exists a unitary operator $L\colon\mathcal{G}_1\rightarrow\mathcal{G}_2$ such that $D(H)=\{x\in D(H_0^{\star})\:;\:LF_1x=F_2x\}$. Here, $F_i\colon D(H_0^{\star})\rightarrow \mathcal{G}_i$ is given by $F_ix=\pr_i(F(x,H_0^{\star}x))$ for $i\in\{1,2\}$, where $\pr_i\colon\mathcal{G}_1\oplus\mathcal{G}_2\rightarrow\mathcal{G}_i$ denotes the canonical projection.
\end{thmA}

\smallskip

Boundary systems generalize the classical concept of boundary triplets on which there exists a considerable amount of literature. We follow the approach of Gorbachuk et al.~\cite{GG} but refer to Wegner \cite{BT} where the skew-symmetric version of the latter results has been outlined explicitly. The notion of boundary triples, as used e.g.~by Br\"uning et al.~\cite{G} coincides with the original, i.e., symmetric, version of boundary triplets in the sense of \cite{GG}. In the remainder, we say that a triple $(\mathcal{G},\Gamma_1,\Gamma_2)$ is a \textit{boundary triplet} for $H_0$ if $\mathcal{G}$ is a Hilbert space and $\Gamma_1$, $\Gamma_2\colon D(-H_0^{\star})\rightarrow \mathcal{G}$ are linear maps such that the map $D(-H_0^{\star})\rightarrow \mathcal{G}\oplus\mathcal{G}$, $x\mapsto (\Gamma_1x,\Gamma_2x)$ is surjective and such that
\begin{equation}\label{BT-1}
(H_0^{\star}x,y)_\mathcal{H}+(x,H_0^{\star}y)_\mathcal{H}=(\Gamma_1x,\Gamma_2y)_\mathcal{G}+(\Gamma_2x,\Gamma_1y)_\mathcal{G}
\end{equation}
holds for all $x$, $y\in D(-H_0^{\star})$. With this notation, the skew-symmetric version of \cite[Theorem III.1.6]{GG}, see \cite[Theorem 4.2]{BT}, reads as follows.

\smallskip

\begin{thmB} Let $H_0$ be skew-symmetric and let $(\mathcal{G},\Gamma_1,\Gamma_2)$ be a boundary triplet for $H_0$. Then $H\supseteq H_0$ is skew-self-adjoint if and only if there exists a unitary operator $L\colon\mathcal{G}\rightarrow\mathcal{G}$ such that $D(H)=\{x\in D(-H_0^{\star})\:;\:(L-1)\Gamma_1x+(L+1)\Gamma_2x=0\}$ and $H=-H_0^{\star}|_{D(H)}$.
\end{thmB}

\smallskip

Let a skew-symmetric operator be given. In Section \ref{SEC:2} we show that then every boundary triplet induces a boundary system in a natural way. If our operator admits a boundary system that induces at least one skew-self-adjoint extension, then it induces a boundary triplet in a natural way. The parametrization of skew-self-adjoint operators via the boundary system (Theorem~A) is then induced by the parametrization via the boundary triplet (Theorem~B), and vice versa, in a simple way. In Section \ref{SEC:2a} we construct for every skew-symmetric operator---thus in particular also for operators that have no skew-self-adjoint extensions---a natural boundary system which leads to an explicit description of at least one maximal dissipative extension. In addition, this boundary system yields an independent proof of the fact that the equality of deficiency indices is equivalent to the existence of skew-self-adjoint extensions. For the convenience of the reader we give a short and direct proof for Theorem B based on the survey \cite{BT} in Section \ref{SEC:3}.

\medskip

\section{Boundary triplets are boundary systems are boundary triplets---or not?}\label{SEC:2}

\smallskip

At first sight, boundary triplets appear to be the special case of boundary systems with $\mathcal{G}_1=\mathcal{G}_2=\mathcal{G}$ and $F_i=\Gamma_i$ for $i\in\{1,2\}$. In \cite[Remark 3.2(a)]{BS} this is pointed out, but there $\Omega$ and $\omega$ are both the standard skew-symmetric forms, and the boundary triplet, that one obtains in the special case, corresponds to the symmetric situation considered in \cite{GG}. For the skew-symmetric case the aforementioned intuition is also correct, but the result is a bit more technical.

\smallskip

\begin{prop}\label{PROP-0} Let $H_0$ be skew-symmetric and let $(\mathcal{G},\Gamma_1,\Gamma_2)$ be a boundary triplet for $H_0$. Then, $(\Omega,\mathcal{G},\mathcal{G},F,\omega)$ is a boundary system for $H_0$, where $\Omega$ is the standard symmetric form, $\omega$ is the standard unitary form, and $F\colon\gr(H_0^{\star})\rightarrow\mathcal{G}\oplus\mathcal{G}$ is given by
\begin{equation}\label{QQ}
F(x,H_0^{\star}x)=\Bigl({\frac{1}{\sqrt{2}}}(\Gamma_1x+\Gamma_2x),{\frac{1}{\sqrt{2}}}(\Gamma_1x-\Gamma_2x)\Bigr).
\end{equation}
The characterization of skew-self-adjoint extensions of $H_0$ via the boundary triplet $(\mathcal{G},\Gamma_1,\Gamma_2)$, given in Theorem B, arises from the characterization via the boundary system $(\Omega,\mathcal{G},\mathcal{G},F,\omega)$, given in Theorem A, by multiplication with $-1$.
\end{prop}
\begin{proof} Firstly, we note that according to the comments behind \cite[Definition 2.3]{BT} the domain $D(-H_0^{\star})$ is equal to $D(H_0^{\star})$. For $x,y\in D(H_0^{\star})$ we use \eqref{QQ} and \eqref{BT-1} to compute
\begin{eqnarray*}
\omega(F(x,H_0^{\star}x),F(y,H_0^{\star}y)) & = & \omega\Bigl(\Bigl({\frac{\Gamma_1x+\Gamma_2x}{\sqrt{2}}},{\frac{\Gamma_1x-\Gamma_2x}{\sqrt{2}}}\Bigr),\Bigl({\frac{\Gamma_1y+\Gamma_2y}{\sqrt{2}}},{\frac{\Gamma_1y-\Gamma_2y}{\sqrt{2}}}\Bigr)\Bigr)\\
& = & \Bigl({\frac{\Gamma_1x+\Gamma_2x}{\sqrt{2}}},{\frac{\Gamma_1y+\Gamma_2y}{\sqrt{2}}}\Bigr)_{\mathcal{G}}-\Bigl({\frac{\Gamma_1x-\Gamma_2x}{\sqrt{2}}},{\frac{\Gamma_1y-\Gamma_2y}{\sqrt{2}}}\Bigr)_{\mathcal{G}}\\
& = & {\frac{1}{2}}\bigl[ (\Gamma_1x,\Gamma_1y)_{\mathcal{G}}+(\Gamma_1x,\Gamma_2y)_{\mathcal{G}}+(\Gamma_2x,\Gamma_1y)_{\mathcal{G}}+(\Gamma_2x,\Gamma_2y)_{\mathcal{G}}\\
& & \hspace{18pt}-(\Gamma_1x,\Gamma_1y)_{\mathcal{G}}+(\Gamma_1x,\Gamma_2y)_{\mathcal{G}}+(\Gamma_2x,\Gamma_1y)_{\mathcal{G}}-(\Gamma_2x,\Gamma_2y)_{\mathcal{G}}\bigr]\\
& = & (\Gamma_1x,\Gamma_2y)_{\mathcal{G}}+(\Gamma_2x,\Gamma_1y)_{\mathcal{G}}\phantom{\frac{1}{2}}\\
& = & (H_0^{\star}x,y)_\mathcal{H}+(x,H_0^{\star}y)_\mathcal{H}\\
& = & \Omega((x,H_0^{\star}x),(y,H_0^{\star}y))\phantom{\frac{1}{2}}
\end{eqnarray*}
which shows \eqref{BS-1}. Let $(y_1,y_2)\in\mathcal{G}\oplus\mathcal{G}$ be given. As $D(-H_0^{\star})\rightarrow \mathcal{G}\oplus\mathcal{G}$, $x\mapsto (\Gamma_1x,\Gamma_2x)$ is surjective, we find $x\in D(H_0^{\star})$ such that $\Gamma_1x=\frac{1}{\sqrt{2}}(y_1+y_2)$ and $\Gamma_2x=\frac{1}{\sqrt{2}}(y_1-y_2)$ hold. Now we compute
$$
F(x,H_0^{\star}x)=\Bigl(\frac{1}{\sqrt{2}}\Bigl(\frac{1}{\sqrt{2}}(y_1+y_2)+\frac{1}{\sqrt{2}}(y_1-y_2)\Bigr),\frac{1}{\sqrt{2}}\Bigl( \frac{1}{\sqrt{2}}(y_1+y_2)-\frac{1}{\sqrt{2}}(y_1-y_2) \Bigr)\Bigr)=(y_1,y_2)
$$
which shows that $F$ is surjective. For the final statement it is enough to observe that \eqref{QQ} implies that $F_1x=\frac{1}{\sqrt{2}}(\Gamma_1x+\Gamma_2x)$ and $F_2x=\frac{1}{\sqrt{2}}(\Gamma_1x-\Gamma_2x)$ hold and that we therefore have
$$
LF_1x=F_2x\;\Longleftrightarrow\;(L-1)\Gamma_1x+(L+1)\Gamma_2x=0
$$
for $x\in D(H_0^{\star})$. This shows that the domains given in Theorem A and Theorem B coincide. Our remarks at the beginning of Section \ref{SEC:1} conclude the proof.
\end{proof}

\smallskip

Our next result is the converse of Proposition \ref{PROP-0} but below we have to assume explicitly that the initial operator $H_0$ has at least one skew-self-adjoint extension. Notice however that this was implicitly also assumed in Proposition \ref{PROP-0}, as the existence of a boundary triplet guarantees that there exist skew-self-adjoint extensions, see Section \ref{SEC:2a}.

\smallskip

\begin{prop}\label{LEM-1} Let $(\Omega,\mathcal{G}_1,\mathcal{G}_2,F,\omega)$ be a boundary system for $H_0$, where $\Omega$ is the standard symmetric form and $\omega$ is the standard unitary form. Let $L_0\colon\mathcal{G}_1\rightarrow\mathcal{G}_2$ be a unitary operator and let $\Gamma_1,\Gamma_2\colon D(-H_0^{\star})\rightarrow\mathcal{G}_1$ be defined via
\begin{equation}\label{Dfn-Gamma}
\Gamma_1x=\frac{1}{\sqrt{2}}(F_1x+L_0^{-1}F_2x)\;\text{ and }\;\Gamma_2x=\frac{1}{\sqrt{2}}(F_1x-L_0^{-1}F_2x).
\end{equation}
Then $(\mathcal{G}_1,\Gamma_1,\Gamma_2)$ is a boundary triplet for $H_0$.
\end{prop}

\begin{proof} We observe that the definitions in \eqref{Dfn-Gamma} yield the two equations
\begin{eqnarray*}
F_1x &=& \frac{1}{2}(F_1x+L_0^{-1}F_2x+F_1x-L_0^{-1}F_2x) = \frac{1}{2}(\sqrt{2}\,\Gamma_1x+\sqrt{2}\,\Gamma_2x) = \frac{1}{\sqrt{2}}(\Gamma_1x+\Gamma_2x)\\
L_0^{-1}F_2x&=&\frac{1}{2}(F_1x+L_0^{-1}F_2x-F_1x+L_0^{-1}F_2x) = \frac{1}{2}(\sqrt{2}\,\Gamma_1x-\sqrt{2}\,\Gamma_2x) = \frac{1}{\sqrt{2}}(\Gamma_1x-\Gamma_2x).
\end{eqnarray*}
The above together with \eqref{BS-1} and the definitions of the standard symmetric resp.~unitary form yield
\begin{eqnarray*}
\bigl(x,H_0^{\star}y\bigr)_{\mathcal{H}}+\bigl(H_0^{\star}x,y\bigr)_{\mathcal{H}}
& = & \Omega\bigl((x,H_0^{\star}x),(y,H_0^{\star}y)\bigr)\\
& = & \omega\bigl(F(x,H_0^{\star}x),F(y,H_0^{\star}y)\bigr)\phantom{\sum}\\
& = & (F_1x,F_1y\bigr)_{\mathcal{G}_1}-\bigl(F_2x,F_2y)_{\mathcal{G}_2}\\
& = & (F_1x,F_1y)_{\mathcal{G}_1}-(L_0^{-1}F_2x,L_0^{-1}F_2y)_{\mathcal{G}_1}\\
& = & \Bigl({\frac{\Gamma_1x+\Gamma_2x}{\sqrt{2}}},{\frac{\Gamma_1y+\Gamma_2y}{\sqrt{2}}}\Bigr)_{\mathcal{G}_1} - \Bigl({\frac{\Gamma_1x-\Gamma_2x}{\sqrt{2}}},{\frac{\Gamma_1y-\Gamma_2y}{\sqrt{2}}}\Bigr)_{\mathcal{G}_1}\\
& = & {\frac{1}{2}}\bigl[(\Gamma_1x,\Gamma_1y)_{\mathcal{G}_1}+(\Gamma_1x,\Gamma_2y)_{\mathcal{G}_1}+(\Gamma_2x,\Gamma_1y)_{\mathcal{G}_1}+(\Gamma_2x,\Gamma_2y)_{\mathcal{G}_1}\\
& & \hspace{18pt}-(\Gamma_1x,\Gamma_1y)_{\mathcal{G}_1}+(\Gamma_1x,\Gamma_2y)_{\mathcal{G}_1}+(\Gamma_2x,\Gamma_1y)_{\mathcal{G}_1}-(\Gamma_2x,\Gamma_2y)_{\mathcal{G}_1}\bigr]\\
& = & (\Gamma_1x,\Gamma_2y)_{\mathcal{G}_1}+(\Gamma_2x,\Gamma_1y)_{\mathcal{G}_1}\phantom{\frac{1}{2}}
\end{eqnarray*}
for $x,y\in D(-H_0^{\star})$. This establishes condition \eqref{BT-1}. For the surjectivity let $y_1, y_2\in\mathcal{G}_1$ be given. By the surjectivity of $F$ we find $x\in D(H_0^{\star})$ such that $F_1x=\frac{1}{\sqrt{2}}(y_1+y_2)$ and $F_2x=\frac{1}{\sqrt{2}}L_0(y_1-y_2)$ hold. Thus, we obtain the following two equations
\begin{eqnarray*}
\Gamma_2x\!\!\!&=&\!\!\!\frac{1}{\sqrt{2}}\bigl(F_1x-L_0^{-1}F_2x\bigr)\;=\;\frac{1}{\sqrt{2}}\biggl(\frac{y_1}{\sqrt{2}}+\frac{y_2}{\sqrt{2}}-\frac{y_1}{\sqrt{2}}+\frac{y_2}{\sqrt{2}}\biggr)\;=\;\frac{1}{\sqrt{2}}\frac{2}{\sqrt{2}}\,y_2 \;=\; y_2\\
\Gamma_1x\!\!\!&=&\!\!\!\frac{1}{\sqrt{2}}\bigl(F_1x-L_0^{-1}F_2x\bigr) \;=\;\frac{1}{\sqrt{2}}\biggl(\frac{y_1}{\sqrt{2}}+\frac{y_2}{\sqrt{2}}+\frac{y_1}{\sqrt{2}}-\frac{y_1}{\sqrt{2}}\biggr) \;=\; \frac{1}{\sqrt{2}}\frac{2}{\sqrt{2}}\,y_1 \;=\; y_1
\end{eqnarray*}
which finish the proof.
\end{proof}

\smallskip

Now we show that under the assumption, that a given boundary system induces at least one skew-self-adjoint extension $H$, the statement of Theorem~A arises from Theorem~B by identifying unitary operators in $L(\mathcal{G}_1,\mathcal{G}_2)$ and unitary operators in $L(\mathcal{G}_1)$ by composition with the inverse of the unitary operator in $L(\mathcal{G}_1,\mathcal{G}_2)$ that corresponds to the extension $H$. The bijection $\Psi$ that appears below originates from Theorem~A, i.e., from the skew-symmetric version of \cite[Theorem III.1.6]{GG}, see also Section \ref{SEC:3}. For the formulation of the next result we use the abbreviations
\[
U(\mathcal{G}_1,\mathcal{G}_2)= \bigl\{L\in L(\mathcal{G}_1,\mathcal{G}_2)\:;\: L \text{ unitary}\bigr\} \;\text{ and }\; U(\mathcal{G})= U(\mathcal{G},\mathcal{G}).
\]
for Hilbert spaces $\mathcal{G}_1$, $\mathcal{G}_2$, and $\mathcal{G}$.

\smallskip

\begin{thm}\label{THM} Let $H_0$ be skew-self-adjoint and let $(\Omega,\mathcal{G}_1,\mathcal{G}_2,F,\omega)$ be a boundary system for $H_0$, where $\Omega$ is the standard symmetric form and $\omega$ is the standard unitary form. Let $L_0\colon\mathcal{G}_1\rightarrow\mathcal{G}_2$ be a unitary operator. Let $(\mathcal{G}_1,\Gamma_1,\Gamma_2)$ be the boundary triplet established in Proposition \ref{LEM-1} and let
$$
\Psi\colon U(\mathcal{G}_1)\longrightarrow\bigl\{H\colon D(H)\subseteq \mathcal{H}\rightarrow\mathcal{H}\:;\:H_0\subseteq H \text{ and } H=-H^\star \bigr\}
$$
be the bijection established by Gorbachuk et al.~\cite{GG} which is defined via
$$
\begin{array}{r@{}l}\vspace{3pt}
\Psi(L)&{}\;=\;-H_0^{\star}|_{D(\Psi(L))}\\
D(\Psi(L))&{}\;=\;\{x\in D(-H_0^{\star})\:;\:(L-1)\Gamma_1x+(L+1)\Gamma_2x=0\}.
\end{array}
$$
Then the bijection
$$
\bar{\Psi}\colon U(\mathcal{G}_1,\mathcal{G}_2)\longrightarrow\bigl\{H\colon D(H)\subseteq \mathcal{H}\rightarrow\mathcal{H}\:;\:H\subseteq H_0^{\star} \text{ and } H=-H^\star \bigr\}
$$
of Schubert et al.~\cite{BS} is given by
$$
\begin{array}{r@{}l}\vspace{3pt}
\bar{\Psi}&(L)\;=\;-\Psi(L_0^{-1}L).\phantom{xxxxxxxxxxxxxxxxxxxx}
\end{array}
$$
\end{thm}
\begin{proof} We define $\bar{\Psi}$ as above, show that it is a bijection, and establish that it gives precisely the correspondence outlined in Theorem A. We observe that 
$$
U(\mathcal{G}_1,\mathcal{G}_2)\longrightarrow U(\mathcal{G}_1),\; L\mapsto L_0^{-1}L
$$
is a bijection. Moreover, the map
$$
\sigma\colon\bigl\{H\colon D(H)\subseteq \mathcal{H}\rightarrow\mathcal{H}\:;\:H_0\subseteq H \;\&\; H=-H^{\star}\bigr\}\longrightarrow\bigl\{H\colon D(H)\subseteq \mathcal{H}\rightarrow\mathcal{H}\:;\:H\subseteq H_0^{\star} \;\&\; H=-H^{\star}\bigr\}
$$
given by
\[
  \sigma H=-H
\]
is well-defined and bijective with inverse given by $\sigma^{-1}=\sigma$. Indeed, if $H\supseteq H_0$ is skew-self-adjoint then $-H$ is also skew-self-adjoint and we have $-H=H^{\star}\subseteq H_0^{\star}$. If, conversely, $H\subseteq H_0^{\star}$ is skew-self-adjoint, then $-H$ is skew-self-adjoint and $-H=H^{\star}\supseteq H_0^{\star\star}=H_0$ since $H_0$ is closed. Combining both facts we conclude that $\bar{\Psi}$ is a bijection. For a unitary operator $L\in U(\mathcal{G}_1,\mathcal{G}_2)$ we compute
$$
\bar{\Psi}(L)=-\Psi(L_0^{-1}L)=H_0^{\star}|_{D(-\Psi(L_0^{-1}L))}
$$
and
\begin{eqnarray*}
D(\bar{\Psi}(L))\;=\;D(-\Psi(L_0^{-1}L))  &=& \{x\in D(-H_0^{\star})\:;\:(L_0^{-1}L-1)\Gamma_1x+(L_0^{-1}L+1)\Gamma_2x=0\}\\
& = & \{x\in D(H_0^{\star})\:;\:L_0^{-1}L(\Gamma_1x+\Gamma_2x)=\Gamma_1x-\Gamma_2x\}\\
&= & \{x\in D(H_0^{\star})\:;\:L_0^{-1}L(\sqrt{2}\,F_1x)=\sqrt{2}\,L_0^{-1}F_2x\}\\
&= & \{x\in D(H_0^{\star})\:;\:LF_1x=F_2x\}
\end{eqnarray*}
which shows that $\bar{\Psi}$ establishes precisely the correspondence given in Theorem A.
\end{proof}

\smallskip

We mentioned already that Proposition \ref{PROP-0} and Theorem \ref{THM} both require, more or less explicitly, that our initial operator $H_0$ has at least one skew-self-adjoint extension. We conclude this section by summarizing the different appearances of the latter condition that we met already up to this point and relate them to the equality of the deficiency indices of $H_0$.

\smallskip

\begin{prop}\label{PROP-P} Let $H_0$ be skew-symmetric. Then the following are equivalent.
\begin{compactitem}\vspace{3pt}
\item[(i)] There exists a skew-self-adjoint extension of $H_0$.\vspace{3pt}
\item[(ii)] The deficiency indices of $H_0$ are equal.\vspace{3pt}
\item[(iii)] There exists a boundary triplet for $H_0$.\vspace{3pt}
\item[(iv)] There exists a boundary system for $H_0$ with $\dim\mathcal{G}_1=\dim\mathcal{G}_2$.
\end{compactitem}
\end{prop}
\begin{proof} (i)\,$\Rightarrow$\,(ii): Let $H\supseteq H_0$ be a skew-self-adjoint extension. Then by \cite[Lemma 5.4]{BT}, which is a skew-symmetric version of the arguments in Schm\"udgen \cite[Section 14]{S}, the equalities
$$
D(H)\dot+\ker(1+H_0^{\star})=D(-H_0^{\star})=D(H)\dot+\ker(1-H_0^{\star})
$$
hold in the sense of a direct but not necessarily orthogonal sum. This means that the deficiency indices of $H_0$ coincide.

\smallskip

(ii)\,$\Rightarrow$\,(iii): \cite[Theorem 5.1]{BT}.

\smallskip

(iii)\,$\Rightarrow$\,(iv): Proposition \ref{PROP-0}.

\smallskip

(iv)\,$\Rightarrow$\,(i): If $\dim\mathcal{G}_1=\dim\mathcal{G}_2$ holds, then there exists a unitary operator $L\colon\mathcal{G}_1\rightarrow\mathcal{G}_2$ that induces via Theorem A a skew-self-adjoint extension of $H_0$.
\end{proof}

\smallskip

Reviewing the results of this section suggests that the notions of boundary triplets and boundary systems are equivalent in the sense that the scope of their extension theories is the same. This is indeed true, but only for the case that skew-self-adjoint extensions are considered. In view of Proposition \ref{PROP-P}, enlarging the class of extensions under consideration means to allow different deficiency indices, and suggests that an extension theory based on boundary systems, then necessarily with $\dim\mathcal{G}_1\not=\dim\mathcal{G}_2$, can yield new insights.

\smallskip


\section{Boundary systems for operators with different deficiency indices}\label{SEC:2a}

\smallskip

From Proposition \ref{PROP-P} it follows that if there is no skew-self-adjoint extension for a skew-symmetric operator, then there is no boundary triplet for this operator either. The situation is different for boundary systems. Below, we construct a natural boundary system for any skew-symmetric operator. We use this boundary system to look for maximal dissipative extensions and to give a short proof for equal deficiency indices being necessary and sufficient for the existence of skew-self-adjoint extensions.

\smallskip

\begin{thm}\label{thm:canBS} Let $H_0$ be skew-symmetric. Let $\mathcal{G}_1:= \ker(1-H_0^\star)$, $\mathcal{G}_2:= \ker(1+H_0^\star)$ and let $P_j\colon D(H_0^\star)\to \mathcal{G}_j$ be the projection for $j\in\{1,2\}$ according to the direct decomposition $D(H_0^\star)=D(H_0)\,\dot+\,\mathcal{G}_1\,\dot+\,\mathcal{G}_2$, see \cite[Lemma 2.5]{BT}. Let $\Omega$ be the standard symmetric form, $\omega$ be the standard unitary form and let $F\colon \gr(H_0^\star)\to \mathcal{G}_1\oplus\mathcal{G}_2$ be defined by
$$
F(x,H_0^{\star}x)=({\sqrt{2}}P_1 x,{\sqrt{2}}P_2x).
$$
Then $(\Omega, \mathcal{G}_1,\mathcal{G}_2,F,\omega)$ is a boundary systems for $H_0$. 
\end{thm}
\begin{proof} The direct decomposition
\[
D(H_0^{\star})=D(H_0)\,\dot+\,\ker(1-H_0^{\star})\,\dot+\,\ker (1+H_0^{\star})
\]
shows immediately that $F$ is surjective. It thus suffices to show that \eqref{BS-1} holds. For this, we note that $H_0^\star P_1x=P_1x$ and $H_0^\star P_2x=-P_2x$ holds for all $x\in D(H_0^\star)$. Let $x,y\in D(H_0^\star)$. Using the orthogonal decomposition above, we have
\[
x=x_0+P_1x+P_2x \;\text{ and }\; x=y_0+P_1y+P_2y
\]
for some $x_0,y_0\in D(H_0)$. It is a lengthy but straightforward computation to verify  \begin{align*}
    \Omega ((x,H_0^\star x),(y,H_0^\star y)) & = ( P_1 x,P_1y)_{\mathcal{H}}+( P_1 x, P_1y)_{\mathcal{H}} -(P_2 x,P_2y)_{\mathcal{H}}-( P_2 x, P_2y)_{\mathcal{H}}
     \\ & =  (\sqrt{2} P_1 x,\sqrt{2}P_1y)_{\mathcal{H}} - (\sqrt{2} P_2 x,\sqrt{2} P_2y)_{\mathcal{H}}
     \\ & = \omega(F(x,H_0^\star x),F(y,H_0^\star y)),
  \end{align*}
  which implies the assertion.
\end{proof}

From Theorem \ref{thm:canBS} and Theorem~A we obtain an alternative proof for the equality of deficiency indices to be a necessary and sufficient condition for the existence of skew-self-adjoint extensions.

\smallskip
 
\begin{cor} Let $H_0$ be skew-symmetric. Then there exists a skew-self-adjoint extension $H\supseteq H_0$ if and only if $\dim\ker(1-H_0^\star)=\dim\ker(1+H_0^\star)$ holds. 
\end{cor}
\begin{proof} Let $(\Omega,\mathcal{G}_1,\mathcal{G}_2,F,\omega)$ be the boundary system constructed in Theorem \ref{thm:canBS}. Then by Theorem~A, skew-self-adjoint extensions exist if and only if there exists a unitary operator from $\mathcal{G}_1$ onto $\mathcal{G}_2$. The latter is equivalent to $\dim\mathcal{G}_1=\dim\mathcal{G}_2$.
\end{proof}

\smallskip

A classical application of the theory of boundary triplets is the characterization of those extensions of a given operator $H_0$ that generate a $\Cnull$-semigroup. For this it is enough if the extension is maximal dissipative, and neither necessary that it is skew-self-adjoint, nor necessary that skew-self-adjoint extensions exist at all. The mere existence of a boundary triplet however requires the latter. The following corollary could be a silver lining indicating that boundary systems might lead to a suitable theory of maximal dissipative extensions for operators with different deficiency indices.

\smallskip

\begin{cor} Let $H_0$ be skew-symmetric and $(\Omega,\mathcal{G}_1,\mathcal{G}_2,F,\omega)$ the boundary system constructed in Theorem \ref{thm:canBS}. Then $H:= H_0^\star|_{D(H_0)\dot{+}\mathcal{G}_2}$ is a maximal dissipative extension of $H_0$. 
\end{cor}
\begin{proof} At first, we show that $H$ is dissipative. Let $x\in D(H)$. Then we use $P_1 x = 0$ to compute 
 \begin{align*}
   2\Re (H_0^\star x,x)_{\mathcal{H}} & =( x,H_0^\star x)_{\mathcal{H}} + (H_0^\star x,x)_{\mathcal{H}}
   \\ & =  \Omega((x,H_0^\star x),(x,H_0^\star x)) 
   \\ & =  \omega(F(x,H_0^\star x),F(x,H_0^\star x)) 
   \\ & = 2 \left((P_1 x,P_1 x)_{\mathcal{H}}-(P_2 x,P_2 x)_{\mathcal{H}}\right)
   \\ & =  -2(P_2 x,P_2 x)_{\mathcal{H}} \leqslant 0.
 \end{align*}
Next, we prove that $H$ is maximal dissipative which is equivalent to $1-H$ being surjective, see e.g.~\cite[Section 7]{BT}. For this, it suffices to prove that $1-H^\star$ is injective since $\ran(1-H)=\ker(1-H^{\star})^{\perp}$ holds and $\ran(1-H)\subseteq\mathcal{H}$ is closed as $H$ is dissipative. Using the boundary system, it is not difficult to see that $D(H^\star)=D(H_0)\dot+\mathcal{G}_1$ and $H^{\star}=-H_0^{\star}|_{D(H_0)\,\dot{+}\,\mathcal{G}_1}$ holds, see Lemma \ref{lem:compadj} below. For $y\in D(H^\star)$ we use the equations above and $P_2y=0$ to obtain
\begin{align*}
-2 \Re ((1-H^\star)y,y)_{\mathcal{H}} & = -2(y,y)_{\mathcal{H}}+2\Re (-H_0^\star y,y)_{\mathcal{H}} 
   \\ & =- 2(y,y)_{\mathcal{H}} -2\bigl((P_1 y,P_1 y)_{\mathcal{H}}-(P_2y,P_2y)_{\mathcal{H}}\bigr)\\
    & = - 2(y,y)_{\mathcal{H}} -2(P_1 y,P_1 y)_{\mathcal{H}}\leqslant -2 (y,y)_{\mathcal{H}},
 \end{align*}
which proves that $1-H^\star$ is injective and concludes the proof.
\end{proof}

\smallskip

\begin{lem}\label{lem:compadj} Let $H_0$ be skew-symmetric, let $\mathcal{G}_1:=\ker(1-H_0^\star)$, $\mathcal{G}_2:=\ker(1+H_0^\star)$ and let $H:=H_0^\star|_{D(H_0)\dot{+}\mathcal{G}_2}$. Then $H^\star=-H_0^\star|_{D(H_0)\dot{+}\mathcal{G}_1}$.
\end{lem}
\begin{proof}
 First of all note that $-H_0\subseteq H\subseteq H_0^\star$. Hence, $H_0\subseteq H^\star\subseteq -H_0^\star$. Let $(\Omega,\mathcal{G}_1,\mathcal{G}_2,F,\omega)$ be the boundary system for $H_0$ introduced in Theorem \ref{thm:canBS}. We have that $y\in D(H^\star)$ holds if and only if 
\[
    \forall\:x\in D(H)\colon (Hx,y)_{\mathcal{H}}=(x,-H_0^\star y)_{\mathcal{H}},
\]
holds. By the definition of $\Omega$ this is equivalent to
 \[
     \forall\:x\in D(H)\colon 0=(x,H_0^\star y)_{\mathcal{H}}+(H_0^\star x,y)_{\mathcal{H}}=\Omega((x,H_0^\star x),(y,H_0^\star y))
 \]
 
being true. By Theorem \ref{thm:canBS}, we deduce that $y\in D(H^\star)$ holds if and only if 
 \[
    \forall\:x\in D(H)\colon (P_1 x,P_1 y)-(P_2 x,P_2 y)=0
 \]
is valid. Since $P_1 x =0$ for all $x\in D(H)$ and $P_2[D(H)]=\mathcal{G}_2$, we infer
 \[
   (y,H_0^\star y) \in \gr(H^\star) \iff y\in D(H_0^\star)\ \&\ P_2 y = 0,
 \]
 which implies the assertion.
\end{proof}

\begin{ex} Let $\mathcal{H}=\operatorname{L}_2(0,\infty)$, $H f = f'$ with $D(H)=\operatorname{W}_{0}^{1,2}(0,\infty)$. Then $H_0$ does not have any skew-self-adjoint extension. However, it is well-known that $H$ is maximal dissipative and generates the right shift semigroup.
\end{ex}


\section{Appendix: direct proof for theorem~b}\label{SEC:3}

\smallskip

In the book \cite{GG} the authors consider extensions of symmetric operators instead of skew-symmetric ones and deal with relations. For the convenience of the reader we give a proof for Theorem B, or more precisely for the bijectivity of the map $\Psi$ that appeared in Theorem \ref{THM}, based on the notation and results outlined in the survey \cite{BT}.

\medskip

\textit{Proof (of Theorem~B). }By \cite[Proposition 4.8]{BT} we know that
$$
\Phi\colon\bigl\{H\colon D(H)\subseteq \mathcal{H}\rightarrow\mathcal{H}\:;\:H_0\subseteq H \text{ maximal dissipative}\bigr\}\longrightarrow\bigl\{L\in L(\mathcal{G})\:;\: L \text{ contraction}\bigr\},
$$
given by $\Phi(H)\colon \Gamma_1x+\Gamma_2x\mapsto \Gamma_1x-\Gamma_2x$, is well-defined and bijective. In particular, we have $\mathcal{G}=\{\Gamma_1x+\Gamma_2x\:;\:x\in D(H)\}$. We fix a skew-self-adjoint extension $H$ of $H_0$. Then we have
$$
D(-H^{\star})=\bigl\{y\in D(-H_0^{\star})\:;\:\forall\;x\in D(H)\colon (\Gamma_1x,\Gamma_2y)_{\mathcal{G}}+(\Gamma_2x,\Gamma_1y)_{\mathcal{G}}=0\bigr\}.
$$
By \cite[Proposition 2.8]{BT} we have $H\subseteq-H_0^{\star}$. In addition, $H_0\subseteq H$ implies that $H^{\star}\subseteq H_0^{\star}$ is true and thus $H=-H^{\star}$ is equivalent to $D(H)=D(-H^{\star})$. We thus obtain that $H=-H^\star$ if and only if
\begin{equation}\label{A}
D(H)=\bigl\{y\in D(-H_0^{\star})\:;\:\forall\;x\in D(H)\colon (\Gamma_1x,\Gamma_2y)_{\mathcal{G}}+(\Gamma_2x,\Gamma_1y)_{\mathcal{G}}=0\bigr\}
\end{equation}
holds. On the other hand we have that $\Phi(H)$ is unitary if and only if
$$
\bigl(\Phi(H)(\Gamma_1x+\Gamma_2x),\Phi(H)(\Gamma_1y+\Gamma_2y)\bigr)_{\mathcal{G}} = \bigl(\Gamma_1x+\Gamma_2x,\Gamma_1y+\Gamma_2y\bigr)_{\mathcal{G}}
$$
holds for all $x,y\in D(H)$. Using the definition of the left hand side of this equality, we obtain after straightforward algebraic manipulations that $\Phi(H)$ is unitary if and only if
\begin{equation}\label{EQ2}
\forall\:x,y\in D(H)\colon (\Gamma_1x,\Gamma_2y)_{\mathcal{G}}+(\Gamma_2x,\Gamma_1y)_{\mathcal{G}}=0
\end{equation}
is valid. Now we claim that $H=-H^\star$ holds if and only if $\Phi(H)$ is unitary.

\medskip

\textquotedblleft{}$\Rightarrow$\textquotedblright{} Assume $H=-H^\star$. Then equation \eqref{A} shows that $(\Gamma_1x,\Gamma_2y)_{\mathcal{G}}+(\Gamma_2x,\Gamma_1y)_{\mathcal{G}}=0$ holds for all $x, y\in D(H)$. This is precisely the condition in equation \eqref{EQ2} and thus $\Phi(H)$ is unitary.

\medskip

\textquotedblleft{}$\Leftarrow$\textquotedblright{} Let $\Phi(H)$ be unitary. Then the condition in equation \eqref{EQ2} implies that in \eqref{A} the inclusion \textquotedblleft{}$\subseteq$\textquotedblright{} holds. It thus remains to establish \textquotedblleft{}$\supseteq$\textquotedblright{}. To show this we fix $y\in D(H_0^{\star})$ such that
$$
(\Gamma_1x,\Gamma_2y)_{\mathcal{G}}+(\Gamma_2x,\Gamma_1y)_{\mathcal{G}}=0
$$
holds for all $x\in D(H)$. Multiplying the latter with two and adding $(\Gamma_1x,\Gamma_1y)_{\mathcal{G}}+(\Gamma_2x,\Gamma_2y)_{\mathcal{G}}$ on both sides yields
$$
2(\Gamma_1x,\Gamma_2y)_{\mathcal{G}}+2(\Gamma_2x,\Gamma_1y)_{\mathcal{G}}+(\Gamma_1x,\Gamma_1y)_{\mathcal{G}}+(\Gamma_2x,\Gamma_2y)_{\mathcal{G}}=(\Gamma_1x,\Gamma_1y)_{\mathcal{G}}+(\Gamma_2x,\Gamma_2y)_{\mathcal{G}}
$$
which we can reorganize to
$$
(\Gamma_1x+\Gamma_2x,\Gamma_1y+\Gamma_2y)_{\mathcal{G}} = (\Gamma_1x-\Gamma_2x,\Gamma_1y-\Gamma_2y)_{\mathcal{G}}.
$$
Since $x$ belongs to $D(H)$ we know that $\Gamma_1x-\Gamma_2x=\Phi(H)(\Gamma_1x+\Gamma_2x)$ holds, which we can plug in on the right hand side of the last equation. Using that $\Phi(H)$ is unitary on the left hand side establishes
$$
\bigl(\Phi(H)(\Gamma_1x+\Gamma_2x),\Phi(H)(\Gamma_1y+\Gamma_2y)\bigr)_{\mathcal{G}} = \bigl(\Phi(H)(\Gamma_1x+\Gamma_2x),\Gamma_1y-\Gamma_2y\bigr)_{\mathcal{G}}
$$
for every $x\in D(H)$. As $\Phi(H)\colon\mathcal{G}\rightarrow\mathcal{G}$ is surjective this yields that
$$
\Phi(H)(\Gamma_1y+\Gamma_2y) = \Gamma_1y-\Gamma_2y
$$
holds. By \cite[Propositions 4.7 and 4.8]{BT} we know
$$
D(H)=\bigl\{y\in D(-H_0^{\star})\:;\:\Phi(H)(\Gamma_1y+\Gamma_2y)=\Gamma_1y-\Gamma_2y\bigr\}
$$
which establishes $y\in D(H)$ as desired.

\smallskip

To conclude the proof it is enough to observe that the map $\Psi$ in Theorem \ref{THM} is the restriction of the inverse of $\Phi$ as given in \cite[Proposition 4.8]{BT}.\hfill\qed

\normalsize

\bibliographystyle{amsplain}
\providecommand{\bysame}{\leavevmode\hbox to3em{\hrulefill}\thinspace}
\providecommand{\MR}{\relax\ifhmode\unskip\space\fi MR }
\providecommand{\MRhref}[2]{%
  \href{http://www.ams.org/mathscinet-getitem?mr=#1}{#2}
}
\providecommand{\href}[2]{#2}

\end{document}